\documentclass[a4paper,10pt]{article}

\usepackage{makeidx}
\usepackage{enumerate}
\usepackage{amsmath,amsthm,amssymb,amscd} 

\def\a{{\alpha}}
\def\b{{\beta}}
\def\g{\gamma}
\def\G{\Gamma}
\def\s{\sigma}
\def\l{\lambda}
\def\k{\kappa}
\def\ep{\varepsilon}
\def\vp{\varphi}
\def\th{\theta}
\def\d{\delta}
\def\Th{\Theta}
\def\DD{\Delta}
\def\vr{\varrho}

\def\P{{\mathbb{P}}}
\def\Q{{\mathbb{Q}}}
\def\C{{\mathbb{C}}}
\def\Z{{\mathbb{Z}}}

\def\A{{\mathbb{A}}}
\def\O{{\mathcal{O}}}
\def\K{{\mathbb{K}}}
\def\I{{{\rm Irr}}}
\def\H{{\mathfrak{H}}}
\def\Sc{{\mathcal{S}}}
\def\OO{{\mathfrak{O}}}
\def\U{{\mathbb{U}}}

\def\n{\mathfrak{n}}
\def\m{\mathfrak{m}}
\def\o{\mathfrak{o}}
\def\p{{\mathfrak{p}}}
\def\P{{\mathfrak{P}}}

\def\i{\infty}
\def\ol{\overline}
\def\t{\times}
\def\ot{\otimes}
\def\bt{\boxtimes}
\def\op{\oplus}
\def\rt{\rtimes}
\def\bs{\backslash}
\def\wt{\widetilde}
\def\ra{\rangle}
\def\la{\langle}
\def\im{\imath}
\def\I{{{\rm Irr}}}
\def\ro{{{\rm ord}}}
\def\Ch{{{\rm ch}}}
\def\dg{{{\rm diag}}}
\def\vol{{{\rm vol}}}

\def\bp{\beta_{\psi}}

\def\1{\mathbf{1}}

\theoremstyle{definition}
\numberwithin{equation}{section}
\newtheorem{thm}{Theorem}[section]

\newtheorem{prop}[thm]{Proposition}
\newtheorem{Cor}[thm]{Corollary}
\newtheorem{lem}[thm]{Lemma}
\newtheorem{Rem}[thm]{Remark}

\newtheorem{Conj}{Conjecture}

\title{\bf L-functions of $S_3(\G_2(2,4,8))$}
\author{\bf TAKEO OKAZAKI\footnote{Supported partially by Grant-in-Aid for JSPS Fellows.}}
\date{}

\begin{document}
\maketitle

\begin{abstract}
The space of Siegel cuspforms of degree $2$ of weight $3$ with respect to the congruence subgroup $\G_2(2,4,8)$ was studied by van Geemen and van Straten in Math. computation. {\bf 61} (1993).
They showed the space is generated by six-tuple products of Igusa $\th$-constants, and all of them are Hecke eigenforms.
They gave conjecture on the explicit description of the Andrianov $L$-functions.  
In J. Number Theory. {\bf 125} (2007), we proved some conjectures by showing that some products are obtained by the Yoshida lift, a construction of Siegel eigenforms.
But, other products are not obtained by the Yoshida lift, and our technique did not work.
In this paper, we give proof for such products. 
As a consequence, we determine automorphic representations of O(6), and give Hermitian modular forms of SU(2,2) of weight $4$.
Further, we give non-holomorphic differential threeforms on the Siegel threefold with respect to $\G_2(2,4,8)$.
\end{abstract}

\section{Introduction}
The Igusa $\th$-constant for characteristic $m= \{a,b,c,d\} \in \Q^4$ is defined by
\begin{eqnarray*}
\th_m(Z)= \sum_{x \in \Z^2} \exp(2\pi i (x + \frac{(a,b)}{2})Z{}^t(x + \frac{(a,b)}{2}) + (x + \frac{(a,b)}{2}){}^t(\frac{(c,d)}{2}))
\end{eqnarray*}
where $Z$ is an element of $\H_2$, the Siegel upper half space of degree $2$.
This is a Siegel modular form of weight $1/2$, hence taking a six-tuple product of Igusa $\th$-constants, we get a Siegel modular form of weight $3$. 
Let $\G(4) \subset Sp_4(\Z)$ be the principal congruence subgroup of level 4, and 
\begin{eqnarray*}
\G_2(2,4,8) = \G(2,4,8) = \Big\{\left[\begin{array}{cc}
1+4A & 4B\\
4C &1+ 4D
\end{array}\right]\in \G(4) \mid 
{\rm diag} (B) \equiv {\rm diag} (C) \equiv 0, {\rm tr}(A)  \equiv 0 \pmod{2} \Big\}.
\end{eqnarray*}
Let $S_3(\G(2,4,8))$ be the space of Siegel cuspforms of weight $3$ with respect to $\G(2,4,8)$. 
Van Geemen and van Straten \cite{vGS}, using the $\th$-embedding $\G(2,4,8) \bs \H_2 \longmapsto \mathbb{P}^{13}$ given in \cite{vG-N},  showed that 
\begin{eqnarray}
S_3(\G(2,4,8)) = \textstyle\sum_{i=1}^7 \C Sp_4(\Z) F_i + \sum_{j=1}^4 \C Sp_4(\Z) g_j, \label{eqn:decS3G(248)}
\end{eqnarray}
where $F_i, g_j$ are six-tuple products of Igusa $\th$-constants.
Further, they showed that each $F_i,g_j$ belongs to an irreducible cuspidal automorphic representation of $PGSp_4(\A)$.
In particular,
\begin{eqnarray*}
g_1(Z) &:=& \theta_{(0,0,0,0)}(2Z)\theta_{(1,0,0,0)}(Z)\theta_{(0,1,0,0)}(Z)
\theta_{(0,0,1,0)}(Z)\theta_{(0,0,1,0)}(Z)\theta_{(0,0,0,1)}(Z),  \\
g_4(Z) &:=& \theta_{(0,0,0,0)}(2Z)\theta_{(1,0,0,0)}(2Z)\theta_{(0,1,0,0)}(2Z)
\theta_{(0,0,1,0)}(Z)\theta_{(0,0,0,1)}(Z)\theta_{(0,0,1,1)}(Z)
\end{eqnarray*}
are Hecke eigenforms.
By computing some eigenvalues of $g_1,g_4$, they gave:
\begin{Conj}[van Geemen and van Straten \cite{vGS}]
Their Andrianov-Evdokimov $L$-functions are written as 
\begin{eqnarray*}
L(s,g_1;{\rm AE}) &=& L(s,\l), \\
L(s,g_4;{\rm AE}) &=& L(s,(\frac{-2}{*}))L(s-1,(\frac{-2}{*}))L(s,\rho_1),
\end{eqnarray*}
up to the Euler factors at $2$.
Here $\l$ is an automorphic character of $\Q(\sqrt{-1},\sqrt{2})$, and $\rho_1$ is an elliptic cuspform of weight $4$.
\end{Conj}
We prove 
\begin{thm}\label{thm:mthm1}
The conjecture is true.
\end{thm}
Hence $g_1$ is in the {\it D-critical case} in the sense of Weissauer \cite{W}.
On the other hand, $g_4$ is $(\frac{-2}{*})$-twist of a {\it Saito-Kurokawa lift} of $\rho_1 \ot (\frac{-2}{*})$.
Van Geemen and Nygaard \cite{vG-N} studied a quotient $Y$ of $\G(2,4,8)\bs \H_2$ and calculated its Hodge numbers $h^{3,0} = h^{2,1} = 1$ and determined the Hasse-Weil zeta function of the third etale cohomology of $Y$.
Combining this with our result in \cite{O2}, we have
\begin{eqnarray*}
L(s,H_{{\rm et}}^3(Y,\Q_2)) = L(s,F_5;{\rm AE})
\end{eqnarray*}
and an interest in the non-holomorphic differential threeform on $Y$.
Motivated to this reason, we give a non holomorphic automorphic form $F_5'$ which defines a differential threeform on $\G(2,4,8) \bs \H_2$.
$F_5$ and $F_5'$ are so-called {\it weak endoscopic lifts} belonging to the same $L$-packet.
All of these representations, weak endoscopic lift, Saito-Kurokawa lift, D-critical case, are familiar with us and appear in the $\th$-correspondence for GSp(4) and GO(4).
Indeed, we will prove the above conjecture in section 3 by using the $\th$-correspondence.
In \cite{O2}, we have proved their conjectures on $L(s,F_i;{\rm AE})$ for $1 \le i \le 6$.
In another work in preparation, by using the $\th$-correspondence for GSp(4) and GO(6), we will show their conjectures for $F_7$ and $g_3$ are true, and give an answer for $g_2$.
Combining all these works, we will complete the proof for their conjectures.

By the way, our result means that there are irreducible automorphic cuspidal representations of GSO(6) corresponding to these familiar representations of GSp(4).
We show the following theorem in section 4 and hence we find holomorphic Hermitian modular forms of SU(2,2) of weight $4$ from the Siegel modular forms of weight $3$.
Notice a holomorphic Hermitian cuspform of SU(2,2) of weight $4$ is canonically identified with a holomorphic differential fourform on a fourfold modular variety, as well as the Siegel cuspforms of degree $2$ of weight $3$.  
\begin{thm}\label{thm:main2}
Let $V$ be an anisotropic six dimensional space defined over $\Q$ with discriminant $d$.
Suppose that a Siegel eigen cuspform $F$ of degree $2$ of weight $3$ appears in the $\th$-correspondence for $GSp(4)$ and $GSO_V$.
Then, there is a holomorphic Hermitian form $\wt{F}$ of $SU_{2,2}(\Q(\sqrt{-d}))$ of weight $4$ with 
\begin{eqnarray}
L(s,\wt{F}) = \zeta(s)L(s,F,(\frac{-d}{*});r_5), \label{eqn:zetaofwtF}
\end{eqnarray}
outside of finitely many bad places.
Here $L(s,\wt{F})$ is the standard Langlands $L$-function of $\wt{F}$ on $SO_{V}(\A) (\simeq SU_{2,2}(\Q(\sqrt{-d})_{\A}))$, and $L(s,F,(\frac{-d}{*});r_5)$ is the $(\frac{-d}{*})$-twist of the $L$-function of degree five ($r_5$ indicates the degree).
If $F$ satisfies the generalized Ramanujan conjecture at almost all good places, then $\wt{F}$ is a cuspform.
\end{thm}
This result is considered as a special case of an analogy of the generic transfer lift from GSp(4) to GL(4) ($\simeq$ GSO(3,3)).
Since the Siegel cuspforms $F_4,F_5,F_6$ which are weak endoscopic lifts, and $g_1$ which is in the D-critical case satisfy the generalized Ramanujan conjecture, we obtain cuspforms $\wt{F}_4,\wt{F}_5,\wt{F}_6$ of $SU_{2,2}(\Q(\sqrt{-1}))$ and $\wt{g}_1$ of $SU_{2,2}(\Q(\sqrt{-2}))$.
On the other hand, by the transfer lift, a weak endoscopic lift of GSp(4) is lifted to a noncuspidal representation of GL(4) (c.f \cite{Asgari-Shahidi}). \vspace{2mm}\\
{\bf ACKNOWLEDGEMENT.}
I express my thanks to Professor R. Salvati Manni for suggesting to solve the conjectures, and to Professor T. Ibukiyama and Professor T. Yamauchi for their kind advice and encouragement. \\
NOTATION.
If $G$ is a reductive algebraic group defined over a number field $F$, $\I(G(\A))$ denotes the set of the equivalence classes of irreducible automorphic representations of $G(\A)$.
At a place $v$ of $F$, let $\I(G(F_v))$ be the set of equivalence classes of irreducible admissible representations of $G(F_v)$.

\section{Preliminaries}
\subsection{$\th$-correspondence for GO(4) and GSp(4)}
In this section, we summarize the $\th$-correspondence for GO(4) and GSp(4).
Let $R$ be a ring, and 
\[
\eta_{n} = \left[\begin{array}{cc}
 & -I_n\\
I_n &
\end{array}\right],
\]
then
\[
GSp_{2n}(R) = \{g \in GL_{2n}(R) \mid {}^{t}g \eta_{n} g = pf(g) \eta_{n} \}, Sp_{2n}(R) = SL_{2n}(R) \cap GSp_{2n}(R).
\]
We call $pf(g)$ the similitude norm of $g$.
Let $X_{/\Q}$ be a four dimensional space defined over $\Q$ with a nondegenerate quadratic form $(\ ,\ )$.
Let $d_X$ be the discriminant of $X$.
We fix the standard additive character $\psi$ on $\Q \bs \A$.
Let $\Sc(X(\Q_v)^n)$ be the space of Schwartz-Bruhat functions of $X(\Q_v)^n$.
The Weil representation $r_v^n$ of $Sp_{2n}(\Q_v) \t O_X(\Q_v)$ defined with respect to $\psi_v$ is the unitary representation on $\Sc(X(\Q_v)^n)$ given by 
\begin{eqnarray}
r_v^n(1,h)\vp_v(x) &=& \vp_v(h^{-1}x), \label{eqn:weilprop1}\\
r_v^n\Big(\left[\begin{array}{cc}
a & 0 \\
0 & {}^ta^{-1}
\end{array}\right],1\Big)\vp_v(x) &=& \chi_X(\det a)|\det a|^{2}\vp_v(xa),\label{eqn:weilprop2} \\
r_v^n\Big(\left[\begin{array}{cc}
I_n & b \\
 0& I_n
\end{array}\right],1\Big)\vp_v(x) &=& \psi_v\Big(\frac{{\rm tr}(b(x,x))}{2}\Big)\vp_v(x), \label{eqn:weilprop3} \\
r_v^n\Big(\left[\begin{array}{cc}
0 & -I_n \\
I_n & 0
\end{array}\right],1\Big)\vp_v(x) &=& \g \vp_v^{\vee}(x). \label{eqn:weilprop4}
\end{eqnarray}
Here $\chi_X(*)= \{*, d_X\}_v$ with the Hilbert symbol $\{*,*\}_v$.
The Weil constant $\g$ is a fourth root of unity depending on the anisotropic kernel of $X,n$ and $\psi$.
The Fourier transformation $\vp^{\vee}$ of $\vp$ is defined by
\begin{eqnarray*}
\vp^{\vee}(x) = \int_{X(\Q_v)^n} \psi_v({\rm tr}(x,y))\vp(y) d y
\end{eqnarray*}
where $d x$ is the self-dual Haar measure.
For our later use, we only consider irreducible cuspidal automorphic representations of $PGSO_X(\A)$.
Let $\s$ be such a representation.
Take $\vp = \ot_v \vp_v$ and $f \in \s$.
Let $r^n = \ot_v r_v^n$.
Let $d h$ be a Haar measure on $SO_X(\Q) \bs SO_X(\A)$.
Then, the integral
\begin{eqnarray}
\th_n(\vp,f)(g) := \int_{SO_X(\Q) \bs SO_X(\A)} \sum_{x \in X(\Q)^n} r_v^n(g,h)\vp(x)f(h) dh \label{eqn:integal}
\end{eqnarray}
is absolutely convergent, and an automorphic form on $Sp_{2n}(\A)$.
By the extended Weil representation as in \cite{Ro2}, $\th_n(\vp,f)(g)$ is extended to an automorphic form on $PGSp_{2n}(\A)$.
We will denote by $\Th_n(\s)$ the space spanned by such automorphic forms on $PGSp_{2n}(\A)$.
We separate the explanation of the $\th$-correspondence for $PGO_X$ to $PGSp_4$, according to $d_X$.

First, suppose that $d_X \in (\Q^{\t})^2$.
In this case, $X$ is isometric to a quaternion algebra $B_{/\Q}$ defined over $\Q$ with the reduced norm.
Let $k$ denote $\Q$, $\Q_v$ or $\A$.
We define the right action $\rho(h_1,h_2)x = h_1^{-1} x h_2$ for $x \in B(k), h_i \in B(k)^{\t}$.
By $\rho$, we have the isomorphisms
\begin{eqnarray}
i_{\rho} &:& GSO_X(k) \simeq B(k)^{\t} \t B(k)^{\t}/\DD k^{\t}, \nonumber \\
& & SO_X(k) \simeq \{(b_1,b_2) \in B(k)^{\t} \t B(k)^{\t} \mid N_{B/k}(b_1) = N_{B/k}(b_2)\}/\DD k^{\t}, \label{eqn:isom1}
\end{eqnarray}
where $\DD k^{\t}$ denotes the diagonal embedding into $B(k)^{\t} \t B(k)^{\t}$.
Hence we can identify a cupidal pair $(\s_1, \s_2)$ of $\I(PB(\A)^{\t})$ with an irreducible cuspidal automorphic representation of $PGSO_X(\A)$.
Further, $\Th_2((\s_1,\s_2)) = \ot_v \th_2((\s_{1v},\s_{2v}))$ is irreducible.
For an cuspidal $\s \in \I(B(\A)^{\t})$, we will denote the Jacquet-Langlands lift of $\s$ by  by $\s^{\rm JL} \in \I(GL_2(\A))$.
If an irreducible cuspidal automorphic representation $\Pi = \ot_v \Pi_v$ of $PGSp_4(\A)$ has
\begin{eqnarray*}
L_S(s,\Pi;{\rm spin}) \big(:= \prod_{v \not\in S} L(s,\Pi_v;{\rm spin}) \big) = L_S(s,\pi_{1})L_S(s,\pi_{2})
\end{eqnarray*}
for a finite set $S$ of places, and a cuspidal pair $(\pi_1,\pi_2)$ of $\I(PGL_2(\A))$, we say $\Pi$ is a weak endoscopic lift of $(\pi_1,\pi_2)$.
In particular, $\Th_2((\s_1,\s_2))$ is a weak endoscopic lift of $(\s_1^{\rm JL},\s_2^{\rm JL})$.
Take $f_1 \in \s_1,f_2 \in \s_2$.
Then (\ref{eqn:integal}) is written as
\begin{eqnarray*}
\th_2(\vp,f_1 \bt f_2)(g) = \int_{SO_X(\Q)\bs SO_X(\A)} \sum_{x \in B(\Q)^2} r^2(g,i_{\rho}(h_1,h_2)) \vp(x)\ol{f}_1(h_1)f_2(h_2) dh_1dh_2.
\end{eqnarray*}
In the case that $B_{/\Q}$ is a definite quaternion algebra, we say $\Th_2((\s_1,\s_2))$ is the Yoshida lift of $(\s_1,\s_2)$, whose archimedean component is holomorphic.
In the case that $B_{/\Q} \simeq M_2(\Q)$, $\Th_2((\s_1,\s_2))$ is globally generic, i.e., there is $F \in \Th_2((\s_1,\s_2))$ which has a nontrivial global Whittaker function.
In general, a global Whittaker function of an automorphic form $F$ on $GSp_4(\A)$ associated to $\psi$ is defined by 
\begin{eqnarray}
W_{F,\psi}(g) = \int_{(\Q \bs \A)^4}\psi(-t+s_4)F(\left[\begin{array}{cccc}
1 &t& & \\
 & 1& & \\
 & & 1& \\
 & & -t& 1
\end{array}\right]
\left[\begin{array}{cccc}
1 & &s_1 &s_2 \\
 & 1&s_2 & s_4\\
 & & 1& \\
 & & &1
\end{array}\right]
g) d t ds_1ds_2ds_4, \label{eqn:derWhitt}
\end{eqnarray}
and decomposed to $\ot_v W_{F,\psi_v}$.
We are going to give $W_{F,\psi_v}$ of $F=\th_2(\vp,f_1\bt f_2)$.
Let 
\begin{eqnarray*}
e = \left[\begin{array}{cc}
 & 1 \\
 & 
\end{array}
\right], \ 
\a = \left[\begin{array}{cc}
1&  \\
 & -1
\end{array}
\right] \in X(\Q) = M_2(\Q).
\end{eqnarray*}
The stabilizer subgroup $Z_{e,\a}(k) \subset SO_X(k)$ of $e,\a$ is isomorphic to 
\begin{eqnarray*}
\bigg\{\big(\left[\begin{array}{cc}
1 & s \\
 & 1
\end{array}
\right], 
\left[\begin{array}{cc}
1 & s \\
 & 1
\end{array}
\right]\big) \mid s \in k \bigg\}
\end{eqnarray*} 
by $i_{\rho}$.
Let $\b_{1,\psi} = \ot_v \b_{1, \psi_v} ,\b_{2,\psi} = \ot_v \b_{2,\psi_v}$ be the Whittaker functions of $f_1,f_2$ with respect to $\psi$.
Then,
\begin{eqnarray} 
W_{F,\psi_v}(g) = \int_{Z_{e_1,\a}(\Q_v) \bs SO_X(\Q_v)} r_v^2(g,i_{\rho}(b_1,b_2))\vp_v(e_{1},\a) \ol{\b}_{1,\psi_{v}}(b_1)\b_{2,\psi_v}(b_2) db_{1}db_{2}. \label{eqn:W3}
\end{eqnarray}

Next, suppose that $d_X \not\in (\Q^{\t})^2$.
Put $L = \Q(\sqrt{d_X})$ with ${\rm Gal}(L/\Q) = \{1,c \}$.
For a quaternion algebra $B_{/\Q} \ot L$, put
\begin{eqnarray}
B(L)^{\pm} = \{b \in B_{/\Q} \ot L \mid b^{\im c} = \pm b \} \label{eqn:defBLpm}
\end{eqnarray}
where $\im$ denotes the main involution of $B$.
Then, $X(\Q)$ is isometric to $B(L)^{+}$ or $B(L)^-$ for a suitable $B_{/\Q}$. 
Define the right action $\rho'(t,h)x = t^{-1}h^{\im c}xh$ for $x \in X, t \in k^{\t}, h \in B(kL)^{\t}$ where $k$ denotes $L,L_{\A}$ or $L_w$, the completion of $L$ at a place $w$ of $L$.
By $\rho'$, we have the isomorphisms
\begin{eqnarray}
i_{\rho'}&:& \{(t,b) \in k^{\t} \t B(Lk)^{\t}\}/\{(N_{Lk/k}(s),s) \mid s \in Lk^{\t} \} \simeq GSO_X(k) \nonumber \\
& &\{(t,b) \mid t^2 = N_{Lk/k}N_{B(Lk)/L}(b) \}/\{(N_{Lk/k}(s),s) \mid s \in Lk^{\t} \} \simeq SO_X(k). \label{eqn:isom2}
\end{eqnarray}
Let $\s \in \I(PB(L_{\A})^{\t})$ be cuspidal.
By $i_{\rho'}$, we can identify $\s$ with an irreducible cuspidal automorphic representation of $PGSO_X(\A)$.
Take $f \in \s$, and $\vp = \ot_v \vp_v \in \Sc(X(\A))$.
Then from (\ref{eqn:integal}), the automorphic forms of $\Th_2(\s)$ are written as
\begin{eqnarray*}
\th_2(\vp,f)(g) = \int_{SO_X(\Q)\bs SO_X(\A)} \sum_{x \in X(\Q)^2} r^2(g,i_{\rho'}(t, b)) \vp(x)f(b) dt d b.
\end{eqnarray*}
In the case that $B_{/\Q} \simeq M_2(\Q)$, $\Th_2(\s)$ is globally generic.
We are going to give the Whittaker function $W_{F,\psi_v}$ of $F = \th_2(\vp,f)$.
Define an additive character $\psi_L$ on $L\bs L_{\A}$ by
\begin{eqnarray*}
\psi_L(z) = \ot_v \psi_v \Big({\rm Tr}_{L_w/\Q_v} (z)\Big),
\end{eqnarray*}
where $w$ denotes a place of $L$ lying over $v$.
Let $X(\Q) = M_2(L)^-$.
Let $e,\a \in M_2(L)^-$ as above.
The stabilizer subgroup $Z_{e,\a}(\A) \subset SO_X(\A)$ is isomorphic to 
\begin{eqnarray}
\bigg\{\big(1,\left[\begin{array}{cc}
1 & s \\
 & 1
\end{array}
\right]\big) \mid s \in \sqrt{d_X} \A \bigg\} \label{eqn:zead}
\end{eqnarray} 
by $i_{\rho'}$.
Let $\bp = \ot_w \b_w$ be the global Whittaker function of $f$ associated to $\psi_L$.
If $L_v/\Q_v$ does not split, the local Whittaker function of $F$ is given by
\begin{eqnarray}
W_{F,\psi_v}(g) =\int_{Z_{e,\a}(\Q_v) \bs SO_X(\Q_v)} r_v^2(g,i_{\rho'}(t,b))\vp_v(e, \a) \b_w(b) dt db \label{eqn:glexpwh}
\end{eqnarray}
where $w$ is the (unique) place lying over $v$.
Otherwise, noting that $GSO_X(\Q_v) \simeq GL_2(L_{w_1}) \t GL_2(L_{w_2})/\DD \Q_v^{\t}$ where $w_1,w_2$ are the two places lying over $v$, we give  
\begin{eqnarray*}
W_{F,\psi_v}(g) =\int_{Z_{e,\a}(\Q_v) \bs SO_X(\Q_v)} r_v^2(g,i_{\rho}(b_1,b_2))\vp_v(e, \a) \ol{\b}_{w_1}(b_1) \b_{w_2}(b_2) db_1db_2.
\end{eqnarray*}
On the other hand, in the case that $B_{/\Q}$ is a definite quaternion algebra, we say $\Th_2(\s)$ is the Yoshida lift of $\s$, whose archimedean component is holomorphic.
Let $\s \in \I(PB(L_{\A})^{\t})$ be cuspidal for a definite quaternion algebra $B_{/\Q}$.
Then, the archimedean component of $\s^{\rm JL}$ is a holomorphic discrete series representation with lowest weight greater than or equal to $2$.
By Blasius \cite{Blasius}, $\s^{\rm JL}$ is tempered.
Then,  Roberts \cite{Ro} says the following for the Yoshida lift $\Th_2(\s)$.
\begin{thm}[THEOREM 8.5 of Roberts \cite{Ro}]\label{thm:Roberts}
Let $\s$ be as above.
Assume that $\s^{\rm JL}$ is not a base change lift from $GL_2(\Q_{\A})$.
Then, for every irreducible constituent $\Pi= \ot_v \Pi_v$ of $\Th_2(\s^{\rm JL})$, there is an irreducible constituent $\Pi' = \ot_v \Pi_v'$ of the Yoshida lift $\Th_2(\s)$ such that $\Pi_v \simeq \Pi_v'$ at every nonarchimedean place.
\end{thm}
Remark that in general $\Th_2(\s), \Th_2(\s^{\rm JL})$ are not irreducible, different from $\Th_2((\s_1,\s_2))$ in the case of $d_X \in (\Q^{\t})^2$.
\subsection{degenerate Whittaker functions}\label{subsec:degW}
For $1 \le r \le 2$, let $P_r=N_rM_r$ be the maximal parabolic subgroup of GSp(4) with 
\begin{eqnarray*}
N_{P_r} &=& \Bigg\{\left[\begin{array}{cccc}
1_r &  & v & {}^tw \\
 & 1_{2-r} & w &  \\
 &  & 1_r &  \\
 &  &  & 1_{2-r}
\end{array}\right]
\left[\begin{array}{cccc}
1_r & u &  &  \\
 & 1_{2-r} &  &  \\
 &  & 1_r &  \\
 &  & -{}^tu  & 1_{2-r}
\end{array}\right] \\
& & \ \ \ \ \ \mid v = {}^tv \in M_r, u, w \in M_{r,2-r}\Bigg\}, \\ 
M_{P_r} &=& 
\Bigg\{\left[\begin{array}{cccc}
z &  &  &  \\
 & a &  & b \\
 &  & \det(g){}^t z^{-1} &  \\
 & c &  & d
\end{array}\right] \mid g = \left[\begin{array}{cc}
a & b \\
c & d
\end{array}\right] \in GSp_{4-2r}, z \in GL_r \Bigg\} \\
&\simeq& GL_r \t GSp_{4-2r},
\end{eqnarray*}
where we understand $GSp_0 = GL_1, GSp_2 = GL_2$.
We write $P_1 = Q$ (resp. $P_2 = P$) and call it Klingen (resp. Siegel) parabolic subgroup.
Then we have two embeddings $e_P,e_Q$ of GL(2) $\t$ GL(1) into $M_{P_r}$ naturally. 
If $E$ is a non-cuspform on $GSp_4(\A)$, then we obtain an automorphic form on $GL_2(\A)$ by the integral:
\begin{eqnarray}
\Phi_{\bullet}(E)(g,z) = {\rm vol}(N_{\bullet}(k) \bs N_{\bullet}(\A))^{-1}\int_{N_{\bullet}(k) \bs N_{\bullet}(\A)}E(n e_{\bullet}(g,z)) d n \label{eqn:Sigelopad1}
\end{eqnarray}
where $\bullet$ indicates $P$ or $Q$.
If a function $W_{\psi}^{\bullet}$ on $GSp_4(\A)$ satisfies 
\begin{eqnarray}
W_{\psi}^{\bullet}(\left[\begin{array}{cccc}
1 &u & & \\
 & 1& & \\
 & & 1& \\
 & & -u&1
\end{array}\right]\left[\begin{array}{cccc}
1 & &x &y \\
 & 1& y& z\\
 & & 1& \\
 & & & 1
\end{array}\right]g) = W_{\psi}^{\bullet}(g) \t \left\{
\begin{array}{ll}
\psi(u) & \mbox{if $\bullet = P$,}\\
\psi(z) & \mbox{if $\bullet = Q$,}
\end{array}\right. \label{eqn:degW}
\end{eqnarray}
we call it a $P$-degenerate (resp. $Q$-degenerate) Whittaker function.
One can derive a $\bullet$-degenerate Whittaker function from an automorphic form $F$ similar to (\ref{eqn:derWhitt}) (but it is vanishing if $F$ is a cuspform).
If $F$ has a $\bullet$-degenerate Whittaker function, then $\Phi_{\bullet}(F)$ is not identically zero.
We define local $\bullet$-degenerate Whittaker function on $GSp_4(\Q_v)$, similarly.
\section{Automorphic forms on $GSp_4(\A)$}
The idea of our proof of Theorem \ref{thm:mthm1} is as follows.
Let $\Pi^{(i)}$ be the automorphic representation of $GSp_4(\A)$ attached to $g_i$. 
From the shape of the conjectured $L$-function, we guess that $\Pi^{(i)}$ is given by a patch of $\th$-lift $\Th = \ot_v \Th_v$.
Here $\Th$ has the conjectured $L$-function, and $\Th_p$ is unramified at $p \neq 2$.
Obviously, $\Th_p$ at $p \neq 2$ has a right $GSp_4(\Z_p)$-invariant vector.
At $p=2$, we really construct a right $\G(2,4,8)_2$-invariant vector of $\Th_2$, where $\G(2,4,8)_2$ is the open closure of $\G(2,4,8)$ in $GSp_4(\Z_2)$.
Hence, we can conclude that $\Pi^{(i)} = \Th$, using (\ref{eqn:decS3G(248)}) and some eigenvalues of $g_i$ calculated by \cite{vGS}.
Hence, the conjecture is proven.
The vector we compute is a local Whittaker function, or local degenerate Whittaker function of $\Th_2$.
These functions are easy to compute and we can construct easily the $\G(2,4,8)_2$-fixed vector.\\
\subsection{Proof for $g_1$, D-critical case}
For an integral ideal $\n$ of an integer ring $\O$, let
\begin{eqnarray*}
\G_0^{(n)}(\n) := \{\g = \left[\begin{array}{cc}
a_{\g} & b_{\g} \\
c_{\g} & d_{\g}
\end{array}\right] \in GSp_{2n}(\O) \mid a_{\g}, b_{\g}, d_{\g} \in M_n(\O), c_{\g} \in M_{n}(\n) \}.
\end{eqnarray*}
First, we would like to show
\begin{prop}\label{prop:defN}
Let $L/\Q$ be a quadratic extension with discriminant $\d_L$. 
Let $\chi_{L}$ be the quadratic character associated to $L/\Q$.
Let $\pi \in \I(PGL_2(L_{\A}))$ be cuspidal with level $\n$.
Let $p$ be a prime which does not split in $L/\Q$ and $\p$ the unique prime ideal of $L$ lying over $p$.
Then, there is $f \in \Th_2(\pi)$ such as
\begin{eqnarray*}
\vr(\g)f = \chi_{L,p}(\det(a_{\g}))f, \ \ \g \in \G_0^{(2)}(p^N\Z_p),
\end{eqnarray*}
where $\vr$ indicates the right translation and 
\begin{eqnarray*}
N = 
\left\{
\begin{array}{ll}
{\ro_{\p}(\d_L)+2^{-1}\ro_{\p}(\n)} & \mbox{if $p$ is ramified and $\ro_{\p}(\n)$ is even,}\\
{\ro_{\p}(\d_L)+2^{-1}(\ro_{\p}(\n)+1)} & \mbox{if $p$ is ramified and $\ro_{\p}(\n)$ is odd,}\\
\ro_{\p}(\n) & \mbox{if $p$ is inert.}
\end{array}\right.
\end{eqnarray*}
\end{prop}
We will prove for the first case with $L=\Q(\sqrt{2})$ and $p=2$.
In other cases, the proof is similar or simpler, and omitted.
For an integral ideal $\n$ of $\o_{\p}$, let  
\begin{eqnarray*}
\G_0'(\n) = \left[\begin{array}{cc}
\o_{\p} & \d_L^{-1} \\
\n & \o_{\p}
\end{array}\right] \cap GL_2(L_{\p}).
\end{eqnarray*}
In the case $\ro_{\p}(\n) = 0$, the proof is easier than the below, so we omit.
Suppose that $\ro_{\p}(\n)$ is a positive (even) integer.
Then, $\pi_{\p}$ is a ramified principal series representation or a supercuspidal representation.
Noting that the level of $\pi$ is $\n$ and the conductor of $\psi$ is $(2\sqrt{2})^{-1}$, we see by the local newform theory for GL(2) that $\pi_{\p}$ has a right $\G_0(\d_L\n)'$-invariant local Whittaker function $\b_2$ associated to $\psi_{L_{\p}}$ such that
\begin{eqnarray}
\b_2(\left[\begin{array}{cc}
1 & z\\
 & 1
\end{array}\right]\left[\begin{array}{cc}
t & \\
 & 1
\end{array}\right]) &=& \left\{
\begin{array}{ll}
\psi_{L_{\p}}(z) & \mbox{if $t \in \o_{\p}^{\t}$,}\\
0 & \mbox{otherwise.}
\end{array}\right. \label{eqn:propb2}\\
\vr(\left[\begin{array}{cc}
 & -1\\
2^{N} &
\end{array}\right])\b_2 &=& \pm \b_2 \label{eqn:propb2-1}
\end{eqnarray}
for $N$ given in Proposition \ref{prop:defN}.
For an integral ideal $\m$ of an integer ring $\O$, let $R_0(\m)$ be the Eichler order of level $\m$, i.e.,
\begin{eqnarray}
R_0(\m)=\{\left[\begin{array}{cc}
a & b\\
c & d
\end{array}\right] \in M_2(\O) \mid c \in \m \}. \label{eqn:Eichlerord}
\end{eqnarray} 
We define the Schwartz-Bruhat function of $M_2(L_{\p})^2$ by 
\begin{eqnarray*}
\vp^{(1)}(x_1,x_2) = \Ch(x_1;R_0(\n))\Ch(x_2;R_0(\n))
\end{eqnarray*}
where $\Ch$ indicates the characteristic function.
We have defined $\vp^{(1)}$ so that 
\begin{eqnarray*}
r_2^2(g,i_{\rho'}(t,h))\vp^{(1)} = r_2^2(g,1)\vp^{(1)}
\end{eqnarray*}
if $(t,h) \in i_{\rho'}^{-1}(SO_X(\Q_2))$ with $h \in \G_0(2^N)$ (see (\ref{eqn:isom2}) for the definition of $i_{\rho'}$), and so that 
\begin{eqnarray}
r_2^2(\g,h)\vp^{(1)} = \chi_{L,2}(\det a_{\g}) r_2^2(1,h)\vp^{(1)} \label{eqn:chrg1}
\end{eqnarray}
if $\g \in \G_0(2^N)$.
Then, we are going to check that the local Whittaker function (\ref{eqn:glexpwh}), constructed from $\b_2$ and $\vp^{(1)}$, is not zero at $g=1$.
Put $\U_2 = i_{\rho'}(\Q_2^{\t} \t \G_0(2^N)) \cap SO_X(\Q_2)$.
The local Whittaker function at $g=1$ is written as
\begin{eqnarray}
\vol(\U_2) \int_{Z_{e,\a}(\Q_2) \bs SO_X(\Q_2) /\U_2} r_v(1, (\ol{t},\ol{h}))\vp_v(e, \a) \b_v(\ol{h}) d (\ol{t},\ol{h}), \label{eqn:expwh}
\end{eqnarray}
where $(\ol{t},\ol{h})$ indicates the projection of $(t,h) \in SO_X(\Q_2)$.
By the Iwasawa decomposition of $GL_2(L_{\p})$, we take the following complete system of representatives for $Z_{e,\a}(\Q_2) \bs SO_X(\Q_2) / \U_2$ (see (\ref{eqn:zead}) for $Z_{e,\a}$):
\begin{enumerate}[i)-type]
 \item $(2^m,\left[\begin{array}{cc}
1 & s\\
 & 1
\end{array}\right] \left[\begin{array}{cc}
2^m & \\
 & 1
\end{array}\right]\left[\begin{array}{cc}
1 & \\
l & 1
\end{array}\right])$ with $s \in \Q_2, m \in \Z$ and $l \in \o_{\p}$ modulo $2^N$,
 \item $(2^{m+\frac{N}{2}}, \left[\begin{array}{cc}
1 & s\\
 & 1
\end{array}\right]\left[\begin{array}{cc}
2^m & \\
 & 1
\end{array}\right]\left[\begin{array}{cc}
 & -1\\
2^N&
\end{array}\right])$.
\end{enumerate}
We observe the contribution of the above types to the integral (\ref{eqn:expwh}).
We will denote
$\rho'(t,h)(e,\a)=(\left[\begin{array}{cc}
a_1 & b_1\\
c_1 & d_1
\end{array}\right],\left[\begin{array}{cc}
a_2 & b_2\\
c_2 & d_2
\end{array}\right])$.
For the i)-type, we calculate    
\begin{eqnarray*}
\rho'(t,h)(e,\a)=\Big(\left[\begin{array}{cc}
2^{-m} l & 2^{-m} \\
-2^{-m} ll^{c} & -2^{-m}l^{c}
\end{array}\right],
\left[\begin{array}{cc}
1+ 2^{-m+1}ls & 2^{-m+1}s \\
-(l+l^c) - 2^{-m+1}ll^{c}s & -1- 2^{-m+1} l^{c}s
\end{array}\right]\Big).
\end{eqnarray*}
Suppose $\rho'(t,h)(e,\a) \in {\rm supp}(\vp^{(1)})$.
Then, we find $m \le 0$ by observing $b_1$.
If $m <0$, then 
\[
\rho'(t,\left[\begin{array}{cc}
1 & 1/4\\
 & 1
\end{array}\right]h)(e,\a)
\]
is also in ${\rm supp}(\vp^{(1)})$.
However, noting 
\begin{eqnarray*}
\b_2(\left[\begin{array}{cc}
1 &1/4 \\
 & 1
\end{array}\right]h) = -\b_2(h),
\end{eqnarray*} 
we conclude the contribution is canceled. 
Hence we can assume $m=0$.
Then, observing $c_1$, we find $l \in \p^N$. 
Observing $b_2$, $s \in 2^{-1}\o_{\p}$.
Conversely, if $m = 0, l \in \p^N$ and $s \in 2^{-1}\o_{\p}$, then $\rho'(t,h)(e,\a) \in {\rm supp}(\vp^{(1)})$.
However, since $\b_2$ is a local newvector which is right $\G_0(\d_L\n)'$-invariant, 
\begin{eqnarray*}
\vr(\left[\begin{array}{cc}
1 & \\
c & 1
\end{array}\right])\b_2 = -\b_2
\end{eqnarray*}
if $c \in \sqrt{2^{-1}}\d_L\n \setminus \d_L\n$.
Using this, one can check easily that the total contribution of i)-type is canceled by this property.
On the other hand, for the ii)-type, we calculate  
\begin{eqnarray*}
\rho'(t,h)(e,\a) = 
\Big(\left[\begin{array}{cc}
0 & 0\\
2^{N-m} & 0 
\end{array}\right],\left[\begin{array}{cc}
-1 & 0\\
-2^{N+1-m}s & 1 
\end{array}\right]\Big).
\end{eqnarray*}
Suppose $\rho'(t,h)(e,\a) \in {\rm supp}(\vp^{(1)})$.
For this, observing $c_1$, we find $m \le 0$ is needed.
By (\ref{eqn:propb2}), we can assume $m=0$.
In this case, observing $c_2$ we have $s \in 2^{-1} \o_{\p}$.
Then, using (\ref{eqn:propb2}), one check easily that the total contribution of this type is not zero.
Thus there is an automorphic form $f \in \Th_2(\pi(\l))$ which has a right $\G_0(2^N)$-{\it semi} invariant local Whittaker function $W_{f,\psi_2}$ such as $W_{f,\psi_2}(1) \neq 0$.
Clearly, $f$ is also right $\G_0(2^N)$-{\it semi} invariant in the sense of Proposition \ref{prop:defN}.
This completes the proof of the Proposition.

Now, we are going to prove the conjecture for $g_1$.
Let $\zeta_8=\frac{\sqrt{2} +\sqrt{-2}}{2}$.
Let $L = \Q(\sqrt{2})$ (resp. $\ K = \Q(\zeta_8)$ with the ring of integers $\o$ (resp. $\OO$).
Let $\p$ (resp. $\P$) be the unique (ramified) prime ideal of $\o$ (resp. $\OO$) lying over the prime ideal $2$ of $\Q$.
Recall the definition of the quasi-character $\l$ of $K_{\A}^{\t}$ in p.870 of \cite{vGS}.
The conductor of $\l$ is $(2) = \P^4$, and 
\begin{eqnarray*}
(\OO/\P^4)^{\t}  = \la \zeta_8 (\bmod{2}) \ra \op \la 1+\sqrt{2} (\bmod{2}) \ra \simeq \Z/4\Z \op \Z/2\Z.
\end{eqnarray*}
The $\P$-component $\l_{\P}$ is determined by 
\begin{eqnarray*}
\l_{\P}(\zeta_8 (\bmod{2})) = 1,\ \ \ \l_{\P}(1 + \sqrt{2} (\bmod{2})) = -1.
\end{eqnarray*}
We define the quasi-character $\nu$ on $L_{\p}^{\t}$ with conductor $\p^3$ by $\nu_{\p}(1 +\sqrt{2} (\bmod{\p^3})) = \sqrt{-1}$, where $(\o/\p^3)^{\t} = \la 1+ \sqrt{2} (\bmod{\p^3}) \ra \simeq \Z/4\Z$.
Then, it holds $\l_{\P} = \nu_{\p} \circ N_{K/L}$.
Let $\pi(\l) = \ot_v \pi(\l)_v$ be the unitary irreducible cuspidal automorphic representation of $GL_2(L_{\A})$ obtained from $\l$.
By Theorem 4.6. (iii) of \cite{JL}, $\pi(\l)_{\p}$ is a principal series representation
\begin{eqnarray}
\pi_{\p} = \pi(\nu_{\p},\nu_{\p}\chi_{K/L,\p}) = \pi(\nu_{\p},\ol{\nu}_{\p}), \label{eqn:pi(l)}
\end{eqnarray}
where $\chi_{K/L,\p}$ is the quadratic character of $L_{\p}$ associated to $K_{\P}/L_{\p}$.

Let $g_0 = {\rm diag}(32,8,4^{-1},1) \in GSp_4(\Q_2)$ and 
\begin{eqnarray*}
\K^{(1)} := g_0^{-1} \G_0^{(2)}(64\Z_2) g_0 = \left[\begin{array}{cccc}
\Z_2 & 4\Z_2 & 64 \Z_2 & 32\Z_2\\
4^{-1}\Z_2 & \Z_2 & 32\Z_2 & 8\Z_2\\
 \Z_2 & 2\Z_2 & \Z_2 & 4^{-1}\Z_2\\
2 \Z_2 & 8\Z_2 & 4\Z_2 & \Z_2
\end{array}\right] \cap Sp_4(\Q_2).
\end{eqnarray*}
By the above discussion, we can take a right $\K^{(1)}$-semi invariant automorphic form $f \in \Th_2(\pi(\l))$ so that $W_{f,\psi_2}(1) \neq 0$.   
By the property of Whittaker function, 
\begin{eqnarray*}
\vr(\left[\begin{array}{cccc}
1 & &s_1 & s_2\\
 & 1&s_2 & \\
 & & 1& \\
 & & &1
\end{array}\right])W_{f,\psi_2}(1) = W_{f,\psi_2}(1) \neq 0
\end{eqnarray*}
for $s_1,s_2 \in \Q_2$.
Hence
\begin{eqnarray}
\int_{\G(2,4,8)_2} \vr(u)W_{f,\psi_2}(1) du \neq 0 \label{eqn:avecons}
\end{eqnarray}
where $\G(2,4,8)_2$ indicates the open closure in $Sp_4(\Z_2)$.
Hence $\Th_2(\pi(\l))$ has a right $\G(2,4,8)_2$-invariant vector. 
Let $B_{/\Q}$ be a definite quaternion algebra which splits outside of $\{\i,2\}$.
Then, $B_{/\Q} \ot L$ splits at every nonarchimedean place of $L$ by the Hasse principle.
The multiple weight of the Hilbert modular forms $\pi(\l)$ over $\Q(\sqrt{2})$ is $(4,2)$.
Hence, there is a (unique) cuspidal $\pi(\l)' \in \I(PB(L_{\A})^{\t})$ corresponding to $\pi(\l)$, i.e., $(\pi(\l)')^{{\rm JL}} = \pi(\l)$.
\begin{thm}\label{thm:g1orb}
The Yoshida lift $\Th_2(\pi(\l)')$ has a right $\G(2,4,8)_2 \t \prod_{p \neq 2} GSp_4(\Z_p)$-invariant vector, as well as the globally generic $\Th_2(\pi(\l))$ of (cohomological) weight $(3,-1)$.
\end{thm}
\begin{proof}
Since $\pi(\l)$ is unramified outside of $\{\p\}$, the place lying over $2$, $\Th_2(\pi(\l))$ is also unramified outside of $\{2\}$.
Here notice the extension $L/\Q$ is also unramified outside of $\{2\}$.
At $\p$, by the above discussion, we find that $\Th_2(\pi(\l))$ has a right $\G(2,4,8)_2$-invariant vector.
Hence $\Th_2(\pi(\l))$ has a right $\G(2,4,8)_2 \t \prod_{p \neq 2} GSp_4(\Z_p)$-invariant vector.
Every base change lift has a pure multiple weight at archimedean places, but $\pi(\l)$'s multiple weight is $(4,2)$.
Hence $\pi(\l)$ is not a base change lift from $GL_2(\Q_{\A})$.
Hence, by Theorem \ref{thm:Roberts}, $\Th_2(\pi(\l)')$ also has a right $\G(2,4,8)_2 \t \prod_{p \neq 2} GSp_4(\Z_p)$-invariant vector.
\hspace*{\fill}$\square$
\end{proof}
Finally, we prove the conjecture on Andrianov-Evdokimov's $L$-function of $g_1$.
Evdokimov in \cite{Ev} defined a $L$-function for a classical Siegel modular form $F$ of degree $2$.
We denote it by $L(s,F;{\rm AE})$ and call Andrianov-Evdokimov $L$-function of $F$.
On the other hand, the spinor $L$-function is defined for generalized Whittaker models of an automorphic form on $GSp_4(\A)$.
Let $S_{F}$ be the finite set of places where $F$ is not $GSp_4(\Z_p)$-invariant.
Suppose $F$ is an eigenform with respect to Evdokimov's Hecke operators outside of $S_F$.
Extend $F$ to an automorphic form $f$ on $GSp_4(\A)$ as in the way p. 121 of \cite{O2}.
Take an irreducible automorphic representation $\pi$ so that $f$ is not orthogonal to $\pi$.
Then it holds 
\[
L_{S_F}(s,\pi;{\rm spin}) = L_{S_F}(s,F;{\rm AE}).
\]
Now, recall that van Geemen and van Straten \cite{vGS} decomposed the space $S_3(\G(2,4,8))$ to eleven irreducible $Sp_4(\Z)$-orbits, where $Sp_4(\Z)$ acts on $S_3(\G(2,4,8))$ in the classical way.
On the other hand, as seen above, the Yoshida lift $\Th_2(\pi(\l)')$ has a right $\G(2,4,8)_2 \t \prod_{p \neq 2} GSp_4(\Z_p)$-invariant vector, hence $\Th_2(\pi(\l)')$ contains one orbit, at least.
Among the eleven $Sp_4(\Z)$ orbits of $F_i,g_j$, we find only one of $g_1$ belongs $\Th_2(\pi(\l)')$, comparing Hecke eigenvalues at $p = 3,5,7$ in the table of section 8 of \cite{vGS}.
Hence we conclude:
\begin{Cor}
The Yoshida lift $\Th_2(\pi(\l)')$ contains the Siegel modular form $g_1$.
The Andrianov-Evdokimov $L$-function of $g_1$ is $L(s,\l)$, outside of $\{2\}$.
The conjecture is true.
\end{Cor}
\begin{Rem}
Noting that $\nu|_{\Q_2^{\t}}$ is ramified, one can see $\th_2(\pi(\l)_{\p})$ is not a generically distinguished representation, in the sense of \cite{Ro2}
From this, one can show $\th_2(\pi(\l)_{\p})$ is not a Klingen parabolically representation.
Further, consulting the Langlands parameters attached to parabolically induced representations due to Roberts and Schmidt \cite{Ro-Sch}, we find that $\th_2(\pi(\l)_{\p})$ is supercuspidal.
\end{Rem}
\subsection{Proof for $g_4$, a Saito-Kurokawa lift}
We extend the Siegel modular form $g_4$ to an automorphic form on $PGSp_4(\A)$ as in the way p. 121 of \cite{O2}, and denote it also by $g_4$. 
We define the automorphic form on $PGSp_4(\A)$ by
\begin{eqnarray*}
g_4^{(-2)}(g) = \chi_{-2}(pf(g))g_4(g)
\end{eqnarray*}
where $\chi_{-2}(*) = (\frac{-2}{*})$, and $pf(g)$ is the similitude norm of $g \in GSp_4(\A)$.
Recall the fact that $g_4$ is a Hecke eigenform, due to \cite{vGS}.
Let $\Pi^{(4)}, \chi_{-2} \Pi^{(4)}$ be the unitary irreducible cuspidal automorphic representations attached to $g_4, g_4^{(-2)}$. 
Observing eigenvalues of $g_4$ in the table of section 8 of \cite{vGS}, we find that $g_4$ does not satisfy the generalized Ramanujan conjecture. 
Indeed
\begin{eqnarray}
|\a_{p1}| = |\a_{p2}| = p^{\frac{3}{2}}, \ \ |\a_{p3}| =p, \ \ |\a_{p4}| = p^2 \label{eqn:noRam}
\end{eqnarray}
for $p = 3,5,7,11,13,17,19$, if we write $\prod_{i=1}^4 (X-\a_{pi})$ the Hecke polynomial of $\Pi_p^{(4)}$.
Then, by a general theory, we can say the following.
\begin{prop}
The automorphic representation $\chi_{-2}\Pi^{(4)}$ is a Saito-Kurokawa lift, i.e., 
\[
L_S(s,\chi_{-2}\Pi^{(4)};{\rm spin}) = \zeta_S(s-\frac{1}{2})\zeta_S(s+\frac{1}{2})L_S(s,\s)
\]
for some cuspidal $\s \in \I(PGL_2(\A))$, and a finite set $S$ of places.
Clearly, $\Pi^{(4)}$ is the $\chi_{-2}$-twist of the Saito-Kurokawa lift $\chi_{-2}\Pi^{(4)}$.
\end{prop}
\begin{proof}
The archimedean component $\Pi_{\i}^{(4)}$ is a holomorphic discrete series representation of weight $(3,3)$.
According to Theorem I. of Weissauer \cite{W}, there is a $GL_4(\ol{\Q}_2)$-valued Galois representation $\rho_{\Pi^{(4)}}$ of ${\rm Gal}(\ol{\Q}/\Q)$, unramified outside of $\{2\}$, such that  
\[
L(s-\frac{3}{2},\Pi^{(4)};{\rm spin}) = L(s,\rho_{\Pi^{(4)}})
\]
outside of $\{ 2\}$.
If we assume that $\Pi^{(4)}$ is not a CAP representation, then $\rho_{\Pi^{(4)}}$ is pure of weight $3$, and the eigenvalues of $\rho_{\Pi^{(4)}}({\rm Frob}_p)$ has absolute value $p^{3/2}$ outside of $\{2\}$.
But, this conflicts to (\ref{eqn:noRam}).
Hence $\Pi^{(4)}$ is a CAP representation, i.e., cuspidal representation associated to a parabolically induced representation.
Recall there are three parabolic subgroups $P,Q$, and Borel subgroup $B$ in Sp(4).
According to Theorem A. of Soudry \cite{Soudry}, every CAP representation associated to $Q$ or $B$ appear in the $\th$-correspondence for GSp(4) and $GO_T$ for a quadratic field $T$ over $\Q$.
Let $\s_T$ be an automorphic representation of $GO_T(\A)$ such as $\Th_2(\s_T) \supset \Pi$ as in Theorem A.
Considering the $\th$-correspondence, one can see that $\Th_2(\s_T)_{\i}$ can not be a holomorphic discrete series representation of weight $(3,3)$.
Hence $\Pi^{(4)}$ is concluded to be a CAP representation associated to the Siegel parabolic subgroup $P$, i.e., some twist of a Saito-Kurokawa lift.
Now, observing the eigenvalues of $g_4$, the assertion is an immediate consequence.
\hspace*{\fill}$\square$
\end{proof}
Let $\s \in \I(PB(\A)^{\t})$ be cuspidal for a quaternion algebra $B_{/\Q}$. 
Let $\1_{B(\A)^{\t}}$ denote the trivial representation of $B(\A)$.
For a $\{\pm 1\}$-valued character $\chi$ of $\Q^{\t} \bs \A^{\t}$, we denote $\chi \s = \s \ot \chi$.
We abbreviate the noncuspidal automorphic representation $\chi \1_{B(\A)^{\t}}$ of $PB(\A)^{\t}$ to $\chi$.
For a pair of noncuspidal $\chi$ and cuspidal $\s$, we consider $\Th_2((\chi, \s))$, similar to cuspidal pairs.
If $B_{/\Q}$ is not split, then $\Th_2((\chi,\s))$ is cuspidal.
In particular, if $B_{/\Q}$ is definite, we call it the Yoshida lift of $(\chi,\s)$.
$\Th_2((\chi, \s))$ is not vanishing, if and only if $L(\frac{1}{2},\chi\s) \neq 0$.
On the other hand, if $B_{/\Q}$ is split, then $\Th_2((\chi,\s))$ is neither cuspidal nor vanishing.
Indeed, one can construct $f \in \Th_2((\chi,\s))$ so that $W_{f,\psi}^P \not\equiv 0$, and hence $\Phi_P(f)$ (defined in (\ref{eqn:Sigelopad1})) is not identically zero.

Now then, we recall the result of Cogdell, Piatetski-Shapiro \cite{Cog-PS} and Schmidt \cite{Sc}.
For a cuspidal $\s = \ot_v \s_v \in \I(PGL_2(\A))$, the global cuspidal Saito-Kurokawa packet ${\rm SK}(\s)$ is the set of irreducible cuspidal automorphic representations of $PGSp_4(\A)$ whose spinor $L$-functions are equal to $\zeta(s-\frac{1}{2})\zeta(s+\frac{1}{2})L(s,\s)$, outside of finite sets of places. 
The local Saito-Kurokawa packet ${\rm SK}(\s_v)$ is the set of $v$ components of the constituents of ${\rm SK}(\s)$.
As explained in p.230-p.233 of \cite{Sc}, the $p$-component $\th_2((1,\s_p)) := \Th_2((1,\s))_p$ is the local Saito-Kurokawa representation of $\s_p$ which is the unique irreducible quotient of the Siegel parabolically induced representation denoted by $|*|^{1/2}\s_{p} \rt |*|^{-1/2}$ in \cite{Sc}, \cite{Ro-Sch}.
Hence $\th_2((\chi_p,\s_p))$ for a quasi-character $\chi_p$ is the $\chi_p$-twist of the local Saito-Kurokawa representation.
One can regard $\th_2((\chi_p,\s_p))$ as the $\chi_p$-twist of $\th_2((1,\chi\s_p))$.
Then, we will see the local, global cuspidal Saito-Kurokawa packet of $\rho_1$ and $\chi_{-2}\rho_1$.
Let $D_{/\Q}$ be a definite quaternion algebra which splits outside of $\{\i, 2\}$.
As seen in section 4 of \cite{O2}, $\rho_1$ corresponds to an irreducible cuspidal automorphic representation of $PD(\A)^{\t}$.
We denote it by $\rho_1'$.
In \cite{O2}, the Siegel modular form $F_1$ is constructed by the Yoshida lift  of $(\1_{D(\A)^{\t}},\rho_1')$.
This means 
\begin{eqnarray*}
\ep(\frac{1}{2},\rho_{1,2}) = 1, L(\frac{1}{2},\rho_1) \neq 0.
\end{eqnarray*}
Since $\chi_{-2}(2) = -1$, we have $\ep(\frac{1}{2},(\chi_{-2}\rho_{1})_2) = -1$.
Since $\chi_{-2}\rho_1$ is unramified outside of $\{2\}$, the root number of $\chi_{-2}\rho_1$ is $-1$.
Hence $L(s,\chi_{-2}\rho_1) = -L(1-s,\chi_{-2}\rho_1)$, and  
\begin{eqnarray*}
L(\frac{1}{2},\chi_{-2}\rho_1) = 0.
\end{eqnarray*}
Let $\s_v$ be the $v$-component of $\rho_1$ or $\chi_{-2}\rho_1$.
Its local Saito-Kurokawa packet is
\begin{eqnarray*}
{\rm SK}(\s_{v}) = 
\left\{
\begin{array}{ll}
\{\pi^{\rm H},\pi^{\rm I}\} & \mbox{at $v = \i$,}\\
\{\th_2((\1_{GL_2},\s_{2})), \th_2((\1_{D_2^{\t}}, \s_{2}')) \} & \mbox{at $v =2$,} \\
\{ \th_2((\1_{GL_2}, \s_{1})) \} & \mbox{otherwise,}
\end{array}\right.
\end{eqnarray*}
where $\s_v'$ indicates the unique admissible irreducible representation of $D_v^{\t}$ corresponding to $\s_{v}$.
$\pi^{\rm H}$ denotes the holomorphic discrete series representation of weight $(3,3)$, and $\pi^{\rm I}$ a cohomologically induced representation, which is neither holomorphic, nor generic.
See section 4 of \cite{Sc} for further explanation.
Since $L(\frac{1}{2},\rho_1) \neq 0$, we conclude by the main lifting theorem of \cite{Sc} that the global cuspidal Saito-Kurokawa packet is composed of the following two constituents:
\begin{eqnarray*}
\ot_{v \neq \i,2} \th_2((\1_{GL_2}, \s_{v})) \ot 
\left\{
\begin{array}{l}
\pi^{\rm H} \ot \th_2((\1_{D_2^{\t}}, \s_2')), \\
\pi^{\rm I} \ot \th_2((\1_{GL_2},\s_2)). 
\end{array}\right. 
\end{eqnarray*}
The Siegel modular form $F_1$ in \cite{vGS} belongs to the holomorphic constituent.
On the other hand, since $L(\frac{1}{2},\chi_{-2}\rho_1) = 0$, we conclude
\begin{eqnarray*}
{\rm SK}(\chi_{-2}\rho_1) = \pi^{\rm H} \ot_{v < \i} \th_2(\1_{GL_2} \bt ((\chi_{-2}\rho_{1})_v).
\end{eqnarray*}
The $p$-component $\th_2((\chi_{-2,p},\rho_{1,p}))$ is the unique irreducible quotient of the Siegel parabolically induced representation $|*|^{1/2}\rho_{1,p} \rt |*|^{-1/2}\chi_{-2,p}$, and $\chi_{-2,p}$-twist of the local Saito-Kurokawa representation of $\chi_{-2,p}\rho_{1,p}$.
The $p$-component $\th_2((\chi_{-2,p},\rho_{1,p}))$ does not have a local Whittaker function.
However, it has a local $P$-degenerate Whittaker function $W_{\psi_p}^P$.
We will compute $W_{\psi_p}^P$.
From $\rho_{1,p}$, we take a right $\G_0^{(1)}(8\Z_p)$-invariant local Whittaker function $\b_p$ with respect to $\psi_p$.
Let $
e' = \left[\begin{array}{cc}
1 & \\ 
 & 
\end{array}\right]$.
Let $Z_{e,e'} \subset SO_X$ be the pointwise stabilizer subgroup of $e,e'$, which is isomorphic to 
\begin{eqnarray*}
\{(\left[\begin{array}{cc}
1 & s\\
 & 1
\end{array}\right],\left[\begin{array}{cc}
1 & \\
 &1
\end{array}\right]) \mid s \in \Q_p \}
\end{eqnarray*}
by $i_{\rho}$.
Then, $W_{\psi_p}^P(g)$ of $\th_2((\chi_{-2,p},\rho_{1,p}))$ is realized as
\begin{eqnarray}
\int_{Z_{e,e'}(\Q_p) \bs GL_2(\Q_p)} r_p^2(g,i_{\rho}(h_1,h_2)) \vp_p(e,e') \chi_{-2,p}(\det(h_2))\b_{p}(h_2) dh_1dh_2. \label{eqn:degWP}
\end{eqnarray}
At $p=2$, we define
\begin{eqnarray*}
& &\vp_2(x_1,x_2) = \chi_{-2}(b_1) \t \\
& &\left\{
\begin{array}{ll}
1 & \mbox{if $\ro_2(a_1) \ge 0, \ro_2(b_1) = 0, \ro_2(c_1), \ro_2(d_1) \ge 3$, $x_2 \in M_2(\Z_2)$,}\\
0 & \mbox{otherwise.}
\end{array}\right.
\end{eqnarray*}
for $x_1= \left[\begin{array}{cc}
a_1 & b_1\\
c_1 & d_1
\end{array}\right], x_2 =\left[\begin{array}{cc}
a_2 &b_2 \\
c_2 &d_2
\end{array}\right] \in M_2(\Q_2)$.
Let
\begin{eqnarray*}
\K^{(4)} = \left[\begin{array}{cccc}
1+ 2^3\Z_2 &\Z_2 & \Z_2&\Z_2 \\
2^3\Z_2 & \Z_2& \Z_2& \Z_2\\
2^6\Z_2 & 2^3\Z_2 &1+2^3\Z_2 &2^3\Z_2 \\
2^3 \Z_2 & \Z_2 & \Z_2 &\Z_2
\end{array}\right] \cap Sp_4(\Z_2).
\end{eqnarray*}
Then, $\vp_2$ is right $\K^{(4)}$-invariant. 
We check (\ref{eqn:degWP}) is not zero at $g=1$ by a direct calculation similar to section 3.1. 
By $g_0' = \dg(2^4,2^3,2^{-1},1)$, 
\begin{eqnarray*}
g_0'^{-1} \K^{(4)}g_0' = 
\left[\begin{array}{cccc}
1+ 2^3\Z_2 & 2\Z_2 & 2^5\Z_2 & 2^4\Z_2 \\
2^2\Z_2 & \Z_2 & 2^4\Z_2 & 2^3\Z_2\\
2\Z_2 & 2^{-1}\Z_2 &1+2^3\Z_2 & 2^2\Z_2 \\
2^{-1}\Z_2 & 2^{-3}\Z_2 & 2\Z_2 & \Z_2
\end{array}\right] \cap Sp_4(\Z_2).
\end{eqnarray*}
Hence there is a local $P$-degenerate Whittaker function $W_2^P \in \th_2((\chi_{-2,2},\rho_{1,2}))$ such that $W_2^P(1) \neq 0$.
Then, similar to (\ref{eqn:avecons}), one can construct a right $\G(2,4,8)_2$-invariant local $P$-degenerate Whittaker function.
Hence $\th_2((\chi_{-2},\rho_{1,2}))$ has a right $\G(2,4,8)_2$-invariant vector.
By the eigenvalues of $g_4$ calculated in \cite{vGS}, we have:
\begin{thm}
The irreducible cuspidal automorphic representation $\Pi^{(4)}$ attached to $g_4$ is a $\chi_{-2}$-twist of a Saito-Kurokawa lift $\rho_1$.
Hence $\chi_{-2}\Pi^{(4)}$ is the unique constituent of global cuspidal Saito-Kurokawa packet of $\chi_{-2}\rho_1$.
\end{thm}
\begin{Cor}
The Andrianov-Evdokimov's $L$-function of $g_4$ is $L(s,\chi_{-2})L(s-1,\chi_{-2})L(s,\rho_1)$.
The conjecture for $g_4$ is true.
\end{Cor}
\begin{Rem}
The local Saito-Kurokawa representation $\chi_{-2}\Pi^{(4)}_2 = \th_2((\1_{GL(2)},\chi_{-2}\rho_{1,2}))$ of $\chi_{-2}\rho_{1,2}$ has the newvector fixed by the paramodular group of level $64$, according to \cite{Ro-Sch}.
It is possible to give the newvector by (\ref{eqn:degWP}) with a suitable Schwartz-Bruhat function.
\end{Rem}
The Siegel modular form $F_2$ given in \cite{vGS} belongs to the holomorphic constituent of the global cuspidal Saito-Kurokawa packet of $\rho_1$, as well as $F_1$.
On the other hand the global cuspidal Saito-Kurokawa packet of $\chi_{-1}\rho_1$ is composed of two constituents, similar to that of $\rho_1$. 
$F_3$ belongs to the holomorphic constituent.
\subsection{Weak endoscopic lifts $F_5$ and $F_6$ (Local Klingen lift)}
Let $\mu$ be the automorphic character of $\Q(\sqrt{-1})$ associated to the CM-elliptic curve $y^2= x^3-x$.
Let $\pi(\mu)$ (resp. $\pi(\mu^3)$) be the irreducible cuspidal automorphic representation of $PGL_2(\A)$ obtained from $\mu$ (resp. $\mu^3$).
Let $\phi_1,\rho_2 \in \I(PGL_2(\A))$ as in the table of p.869. 
Then, 
\begin{eqnarray*}
\phi_1 = \pi(\mu), \ \ \rho_2 = \pi(\mu^3).
\end{eqnarray*}
In \cite{O2}, we showed that there are cuspidal automorphic representations   corresponding to $\phi_1, \rho_1$ of a definite quaternion algebra which splits outside of $\{\i, 2\}$.
We denote them by $\phi', \rho_1'$.
By THEOREM 8.5 of Roberts \cite{Ro}, we find that all the weak endoscopic lifts of $(\phi_1,\rho_2)$ (resp. $(\chi_{2}\phi_1,\chi_2\rho_2)$) are $\{ \Th_2(\phi_1,\rho_2), \Th_2(\phi_1',\rho_2') \}$ (resp. $\{ \Th_2(\chi_{-2}\phi_1,\chi_{-2}\rho_2), \Th_2(\chi_{-2}\phi_1',\chi_{-2}\rho_2') \}$).
Let $F_5,F_6$ be the Siegel modular forms as in \cite{vGS}.
In \cite{O2}, we realize $F_5$ (resp. $F_6$) by the Yoshida lift $(\Th_2((\phi_1',\rho_2'))$ (resp. $\Th_2((\chi_{-2}\phi_1',\chi_{-2}\rho_2'))$).
In this section, we show that $\Th_2((\phi_1,\rho_2)$ and $\Th_2((\chi_{-2}\phi_1,\chi_{-2}\rho_2))$ have right $\G(4,8)_2$-invariant vectors, where 
\[
\G(4,8) = \{\left[\begin{array}{cc}
A & B\\
C & D
\end{array}\right] \in \G(4) \subset Sp_4(\Z) \mid 
{\rm diag} (B) \equiv {\rm diag} (C) \equiv 0 \pmod{8} \}.
\]
This means there are non-holomorphic differential forms corresponding to these vectors on the Siegel threefold related to $\G(4,8)$.
Similarly, we would like to explain that both of $\Th_2((\phi_1',\rho_2'), \Th_2((\chi_{-2}\phi_1',\chi_{-2}\rho_2'))$ have right $\G(4,8)$-invariant vectors, similarly.
This is another proof for the conjecture 
\[
L(s,F_5;{\rm AE}) = L(s,\mu)L(s,\mu^3), \ \ L(s,F_6;{\rm AE}) = L(s,\mu\chi_{-2} \circ N_{\Q(\sqrt{-1})/\Q})L(s,\mu^3\chi_{-2} \circ N_{\Q(\sqrt{-1})/\Q})
\]
of \cite{vGS}.
Then, we start a general setting.
Let $(\s_1,\s_2)$ be a cuspidal pair of $\I(PGL_2(\A))$.
We call the $v$-component $\th_2((\s_{1v},\s_{2v}))$ of $\Th_2((\s_1,\s_2))$ the local $\th$-lift of $(\s_{1v},\s_{2v})$, which is determined only by the local data $(\s_{1v},\s_{2v})$.
If $\s_{1v} = \s_{2v}$, we call $\th_2((\s_{1v}, \s_{1v}))$ the {\it `local generic Klingen lift of $\s_{1v}$'}.
Further, suppose that $\s_{1v}$ is square-integrable.
In this case, there is an irreducible admissible representation $\s_{1v}'$ of the division quaternion algebra over $\Q_v$.
Let $B_{/\Q}$ be a quaternion algebra which does not split at $v$. 
Take an irreducible cuspidal automorphic representation $\s'$ of $PB(\A)^{\t}$ so that $\s_v' \simeq \s_{1v}'$.
We define local $\th$-lift $\th_2((\s_{1v}',\s_{1v}'))$ and call the {\it `local non generic Klingen lift of $\s_{1v}$'}, similarly.
First, we show the following general results.
\begin{prop}\label{prop:noncusp}
Let $\pi$ be an irreducible cuspidal automorphic representation of $PGL_2(\A)$ with level $N$.
Let $\Phi_Q$ be the Siegel operator defined in section \ref{subsec:degW}.
Then, $\Phi_Q(\Th_2((\pi,\pi)))|_{GL_2(\A)} = \pi$.
\end{prop}
\begin{proof}
From $\pi$, take a newform $f$, which is right $\G_0^{(1)}(N)$-invariant.
For $\vp \in \Sc(M_2(\A)^2)$, we consider $F= \th_2(\vp,f \bt f)$.
Put a Schwartz-Bruhat function $\vp^2(x) = \vp(0,x)$ of $M_2(\A)$.
Then, we have 
\begin{eqnarray}
\Phi_Q(F)(g,z) = \th_1(\vp^2,f \bt f)(g). \label{eqn:Siegeldown}
\end{eqnarray}
and 
\begin{eqnarray}
W_{F,\psi}^Q(1) = \int_{SO_X(\Q) \bs SO_X(\A)} r^1(1,i_{\rho}(b_1,b_2))\vp^2(1)\ol{f}(b_1)f(b_2) db_1db_2. \label{eqn:SL2whth}
\end{eqnarray}
The stabilizer subgroup of $1 \in M_2(\Q)$ is $i_{\rho}(\{(b_1,b_1) \mid b_1 \in GL_2(\A) \})$.
We set $\vp_p^2(x) = \Ch(x;R_0(N\Z_p))$ and $\vp_{\i}$ suitably, where $R_0(N\Z_p)$ is the Eichler order of level $N\Z_p$ defined in (\ref{eqn:Eichlerord}).
Then, noting that the Petersson inner product of $f$ and $f$ is not vanishing and that $f$ is right $\G_0(N)^{(1)}$-invariant, one can see easily that (\ref{eqn:SL2whth}) is not zero.
Hence, $\Phi_Q(\Th_2((\pi,\pi)))$ is not vanishing, that is, $\Th_2((\pi,\pi))$ is not cuspidal.
Further, it is easy to see $\Phi_Q(\Th_2((\pi,\pi)))|_{GL_2(\Q_p)} = \pi_p$ at $p \nmid N$.
Hence, by the strong multiplicity one theorem for $GL_2(\A)$, we get the assertion.
\hspace*{\fill}$\square$
\end{proof}
\begin{Rem}
One can see $\Th_2((\pi,\pi))$ is not cuspidal by the result of Kudla-Rallis \cite{K-R}.
Let $S= \{v \mid \mbox{$v = \i$ or $\pi_v$ is ramified} \}$.
The partial $L$-function $
L_S(s,\Th_2((\pi,\pi));r_5)$ of degree five is $\zeta_S(s)L_S(s,\pi \t \pi)$, and has a double pole at $s=1$.
But, according to \cite{K-R}, the partial $L$-function of a cuspidal representation of $GSp_4(\A)$ does not have a double pole.
Hence $\Th_2((\pi,\pi))$ is not cuspidal. 
\end{Rem}
\begin{Cor}\label{cor:degQgnk}
Let $\pi_p$ be an irreducible admissible representation of $PGL_2(\Q_p)$.
Let $n$ be the order of the level of $\pi_{p}$ at $p$ and 
\begin{eqnarray*}
Kl_p(n) = \left[\begin{array}{cccc}
\Z_p & p^n\Z_p &\Z_p &\Z_p \\
\Z_p & \Z_p& \Z_p &\Z_p \\
\Z_p & p^n\Z_p &\Z_p & \Z_p \\
p^n \Z_p& p^n \Z_p& p^n \Z_p & \Z_p
\end{array}\right] \cap GSp_4(\Z_p).
\end{eqnarray*}
Then, the local generic Klingen lift $\th_2((\pi_p,\pi_p))$ has a right $Kl_p(n)$-invariant local $P$-degenerate Whittaker function $W_{\psi_p}^Q$ such that $W_{\psi_p}^Q(1) \neq 0$.
\end{Cor}
\begin{proof}
Take an irreducible cuspidal automorphic representation $\tau$ of $PGL_2(\A)$ so that $\tau_p = \pi_p$.
We consider $\Th_2((\tau,\tau))$.
Let $f \in \tau, \vp_p^2 \in \Sc(M_2(\Q_p))$ be the same as in the proof of Proposition \ref{prop:noncusp}.
Let
\begin{eqnarray*}
\vp_p(x_1,x_2) = \Ch(x_1;M_2(\Z_p))\vp_p^2(x_2).
\end{eqnarray*}
Notice $\vp_p(0,x_2) = \vp_p^2(x_2)$.
By the properties of $r_p^2$ in (\ref{eqn:weilprop2}) $\sim$ (\ref{eqn:weilprop4}), one can check that $W_{F,\psi}^Q$ is right $Kl_p(n)$-invariant.
Then, we see as in the proof of Proposition \ref{prop:noncusp} that $W_{F,\psi}^Q(1) \neq 0$ by choosing a suitable $\vp_v \in \Sc(M_2(\Q_v))$ outside of $\{p\}$.
Hence, $W_{F,\psi}^Q|_{GSp_4(\Q_p)}$ is the desired local $Q$-degenerate Whittaker function.
\hspace*{\fill}$\square$
\end{proof}
\begin{Cor}\label{cor:degQngnk}
Let $\pi_p,n$ be the same as above, and 
\begin{eqnarray*}
Kl_p'(n) = \left[\begin{array}{cccc}
\Z_p & p^n\Z_p &\Z_p &\Z_p \\
\Z_p & \Z_p& \Z_p &\Z_p \\
p\Z_p & p^n\Z_p &\Z_p & \Z_p \\
p^n \Z_p& p^n \Z_p& p^n \Z_p & \Z_p
\end{array}\right] \cap GSp_4(\Z_p).
\end{eqnarray*}
Suppose that $\pi_p$ is square-integrable.
Then, the local nongeneric Klingen lift $\th_2((\pi_p',\pi_p'))$ has a right $Kl_p'(n)$-invariant $W_{\psi_p}^Q$ such that $W_{\psi_p}^Q(1) \neq 0$.
\end{Cor}
\begin{proof}
The proof is similar to Corollary \ref{cor:degQgnk}.
However, notice that the characteristic function of the maximal order of the division quaternion algebra over $\Q_p$ is right $\G_0^{(1)}(p)$-invariant, not $SL_2(\Z_p)$-invariant, with respect to the action of $SL_2(\Q_p)$.
By this reason, $Kl_p'(n)$ is smaller than $Kl_p(n)$ as above.
\hspace*{\fill}$\square$
\end{proof}
Then, we apply above results to our aim.
For this, we need:
\begin{lem}\label{lem:supB}
The $v$-component $\phi_{1,v}$ at $v=2$ is supercuspidal and equivalent to $\rho_{2,2}$.
\end{lem}
\begin{proof}
As seen in \cite{O2}, $\phi_1$ is obtained by the Jacquet-Langlands lift.
Hence $\phi_{1,2}$ is square-integrable.
The level of special representations of $PGL_2(\Q_p)$ is $p$, $\phi_{1,2}$, of level 32, is not a special representation.
Hence the first claim.
Let $\p$ be the place of $\Q(\sqrt{-1})$ lying over $2$.
We normalize $|\mu_{\p}| =1$.
Then, $\mu_{\p}$ is $\{\pm 1, \pm\sqrt{-1} \}$-valued on $\o_{\p}^{\t}$.
Hence $\mu_{\p} = \ol{\mu}_{\p}^3$ on $\o_{\p}^{\t}$.
Since the central character of $\pi(\mu_{\p})$ is trivial, we have
\[
\phi_{1,2} = \pi(\mu_{\p}) = \pi(\ol{\mu}_{\p}^3)= \ol{\pi(\mu_{\p}^3)} = \ol{\rho}_{2,2} = \rho_{2,2}.
\]
This completes the proof.
\hspace*{\fill}$\square$
\end{proof}
Then we describe our final result in this section.
\begin{thm}
The globally generic weak endoscopic lift $\Th_2((\phi_1,\rho_2))$ (resp. $\Th_2((\chi_{-2}\phi_1,\chi_{-2}\rho_2))$)  of (cohomological) weight $(3,-1)$ has a right $\G(4,8)$-invariant vector.
Similarly, the Yoshida lift $\Th_2((\phi_1',\rho_2'))$ (resp. $\Th_2((\chi_{-2}\phi_1,\chi_{-2}\rho_2'))$) has a right $\G(4,8)$-invariant vector.
In particular, $F_5 \in \Th_2((\phi_1',\rho_2'))$ and $F_6 \in \Th_2((\chi_{-2}\phi_1,\chi_{-2}\rho_2'))$.
Clearly, $L(s,F_5;{\rm AE}) = L(s,\phi_1)L(s,\rho_2)$, $L(s,F_6;{\rm AE}) = L(s,\chi_{-2}\phi_1)L(s,\chi_{-2}\rho_2)$, outside of $\{2\}$.
\end{thm}
\begin{proof}
Since the level of $\chi_{-2}\phi_1,\chi_{-2}\rho_2$ is 64, $\Th_2((\phi_1,\rho_2))$ has a right $Kl_2(6)$ invariant $W_{\psi_2}^Q$ by Corollary \ref{cor:degQgnk}.
$Kl_2(6)$ is isomorphic to the congruence subgroup 
\begin{eqnarray*}
\left[\begin{array}{cccc}
\Z_2 & 2\Z_2& 2^5\Z_2& 2^4\Z_2 \\
2^{-1}\Z_2 & \Z_2 & 2^4\Z_2&2^3\Z_2 \\
2\Z_2 & 2^2\Z_2 & \Z_2 &2^{-1}\Z_2 \\
2^2\Z_2 & 2^3\Z_2&2\Z_2 & \Z_2
\end{array}\right] \cap GSp_4(\Q_2).
\end{eqnarray*}
Then, by using the property of $W_{\psi_2}^Q$, one can construct a right $\G(4,8)$-invariant vector, similar to the discussion in front of Theorem \ref{thm:g1orb}.
The proof for the other weak endoscopic lifts are similar to this.
\hspace*{\fill}$\square$
\end{proof}
\begin{Rem}
If $\pi_p$ is a supercuspidal representation, then the local generic Klingen lift of $\pi_p$ is the constituent of a parabolically induced  representation which is denoted by $\tau(S,\pi_p)$ in \cite{Ro-Sch}.
On the other hand, the local nongeneric Klingen lift is $\tau(T,\pi_p)$.
\end{Rem}
\section{Automorphic forms on $SU_{2,2}(K_{\A})$}
In this section, we will consider the $\th$-correspondence for GSO(6) and GSp(4) and that for U(4) and U(2,2).
Let $V$ be an anisotropic six dimensional space defined over $\Q$ with discriminant $d_V$.
Let $K =\Q(\sqrt{-d_V})$.
Let  
\begin{eqnarray*}
U_{2m}(K) = \{g \in GL_{2m}(K) \mid ^{t}\ol{g}I_{2m}g = I_{2m} \},\ \ 
U_{m,m}(K) = \{g \in GL_{2m}(K) \mid ^{t}\ol{g}\eta_{m}g = \eta_{m} \}
\end{eqnarray*}
and $GU_{2m}(K), GU_{m,m}(K)$ be the similitude groups and $SU_{2m}(K) = SL_{2m}(K) \cap U_{2m}(K), SU_{m,m}(K) = SL_{2m}(K) \cap U_{m,m}(K)$.
Let $\s$ be an irreducible cuspidal automorphic representation of $GSO_V(\A)$.
By the Weil representation of $Sp_4 \t O_V$, we can consider the $\th$-lift of $\s$ to $GSp_4(\A)$, similar to section 2.
We denote the $\th$-lift of $\s$ by $\Th_{2}(\s)$. 
On the other hand, we can obtain an injective isomorphism $J$ of $GU_4(K)$ to $GSO_V$ from Theorem 16 of \cite{Knus}. 
In particular, $J: SU_{4}/\{\pm 1\} \simeq SO_V$.
By $J$, we can obtain an automorphic representation of $GU_4(K_{\A})$.
We denote it also by $\s$.
Then, by the Weil representation of $U_4(K) \t U_{2,2}(K)$, we can consider the $\th$-lift of $\s|_U$ to $U_{2,2}(K_{\A})$ and denote the $\th$-lift by $\Th_{2,2}(\s|_U)$. 
First, we show the following.
\begin{prop}
Let $\s \in \I(GSO_V(\A))$ be cuspidal.  
If $\Th_2(\s) \neq 0$, then  $\Th_{2,2}(\s|_U) \neq 0$. 
\end{prop}
\begin{proof}
Let $r^2 = \ot_v r_v^2$ be the Weil representation of Sp(2) $\t$ $O_V$.
Take a nonzero $F \in \Th_2(\s)$.
The Fourier coefficient $F_T$ associated to $T ={}^tT >0$ is written as 
\begin{eqnarray*}
F_T(1) = \int_{Z_{x_1,x_2}(\A) \bs SO_V(\A)} r^2(1,h)\vp(x_1,x_2) f(h) dh
\end{eqnarray*} 
by some $f \in \s$ and a pair of $x_1,x_2 \in V(\Q)$ whose Gramian is $T$.
Here $Z_{x_1,x_2} \subset SO_V$ indicates the pointwise stabilizer subgroup of $x_1,x_2$, and $\vp$ is a Schwartz-Bruhat function of $V(\A)^2$.
Hence,
\begin{eqnarray*}
\int_{Z_{x_1,x_2}(\Q) \bs Z_{x_1,x_2}(\A)} f(zh) dz \neq 0
\end{eqnarray*}
for some $h \in SO_V(\A)$.
Notice that $Z_{x_1,x_2}$ is isomorphic to an orthogonal group of rank four.
Hence, by the isomorphism $J$, there is a subgroup $U_W \simeq U_2(K)$ embedded into $Z_{x_1,x_2}(\A) \simeq O_4(\A)$, such that
\begin{eqnarray*}
\int_{U_W(K) \bs U_W(K_{\A})} f(zh) dz \neq 0.
\end{eqnarray*}
Let $V_U$ be the Hermitian vector space on which $U_4(K)$ acts.
Notice that $U_W$ stabilizes a pair $y_1,y_2 \in V_U$ such as
\[
Y:= \left[\begin{array}{cc}
\la y_1,y_1 \ra & \la y_1,y_2 \ra \\
\la y_2,y_1 \ra & \la y_2,y_2 \ra
\end{array}\right] > 0,
\]
where $\la *,* \ra$ denotes the Hermite form of $V_U$.
Similar to $F \in \Th_2(\s)$, also for $F' \in \Th_{2,2}(\s|_U)$, we write 
\begin{eqnarray*}
F'_Y(1) = \int_{U_W(K) \bs U_4(K)} r_U(1,h)\phi(y_1,y_2) f(h) dh
\end{eqnarray*}
where $r_U$ is the Weil representation of $U_{2,2} \t U_4$ and $\phi \in V_U(K_{\A})^2$.
It is possible to choose $\phi$ so that $F_{Y}'(1) \neq 0$ (c.f. Concluding Remarks in \cite{Y2}).
Hence the assertion.
\hspace*{\fill}$\square$
\end{proof}
By this proposition, we find the existence of $\wt{F}$ in Theorem \ref{thm:main2}.
Let $S_{\s}$ be the finite set of places of $\Q$ consisting of \\
(i) the archimedean place and all places $v$ where $K_v/\Q_v$ is ramified.\\
(ii) all finite places at which $\s_v$ is ramified.\\
By Kudla \cite{Kudla}, we calculate the standard Langlands $L$-function of $\s|_{SO_V}$ (of degree six) is equal to $\zeta(s)L(s,F_5,\chi_{K};r_5)$ outside of $S_{\s}$.
Here $L(s,F_5,\chi_{K};r_5)_p$ is the $\chi_{K}$-twist of $L(s,F_5;r_5)$, as usual.
It is easy to see that 
\[
L_{S_{\s}}(s,\Th_{2,2}(\s|_U)|_{SO_X}) = L_{S_{\s}}(s,\s|_{SO_X}).
\]
Hence (\ref{eqn:zetaofwtF}).
Finally, we show the last assertion of the Theorem.
For this, we consider the standard Langlands $L$-function of a noncuspidal automorphic representation $\tau$ of $SO_V(\A)$.
Recall the Siegel operator $\Phi$ for Hermitian modular forms on $SU_{2,2}(K_{\A})$.
Let
\begin{eqnarray*}
N_{1} &=& \Bigg\{\left[\begin{array}{cccc}
1 &  & v & w \\
 & 1 & \ol{w} &  \\
 &  & 1 &  \\
 &  &  & 1
\end{array}\right]
\left[\begin{array}{cccc}
1 & u &  &  \\
 & 1 &  &  \\
 &  & 1 &  \\
 &  & -u  & 1
\end{array}\right] \mid v \in \Q, u, w \in K \Bigg\}, \\ 
M_1 &=& \{\left[\begin{array}{cccc}
z^{c}\a & & z^{c}\b& \\
 & t^{-1}z& & \\
z^{c}\g & & z^{c}\d & \\
 &  & &t z
\end{array}\right] \mid \left[\begin{array}{cc}
\a & \b\\
\g & \d
\end{array}\right] \in SU_{1,1}(K), z \in K^1, t \in \Q^{\t} \}
\end{eqnarray*}
where $K^1 = \{z \in K^{\t} \mid N_{K/\Q}(z) = 1 \}$. 
We embed $SU_{1,1}(K) \t K^1 \t \Q^{\t}$ into $M_1$, naturally.
A holomorphic Hermitian modular form $F$ is a noncuspform, if and only if 
\begin{eqnarray*}
\Phi(F)(g,t,z;h) = {\rm vol}(N_{1}(k) \bs N_{1}(\A))^{-1}\int_{N_{1}(K) \bs N_{1}(K_{\A})}F(n (g,t,z)h) d n 
\end{eqnarray*}
is not a zero function of $(g,t,z)$ at some $h \in SU_{2,2}(K_{\A})$ (c.f. \cite{Krieg}).
Then, let $E$ be a holomorphic noncuspform of weight $\k$.
Imitating the method of Zharkovskaya \cite{Zhar}, or observing each local Jacquet module with respect to $N_1$ at a good place, we can write, outside of finitely many bad places, 
\begin{eqnarray}
L(s,E) = L(s,\s_1)L(s,\s_1,\chi)L(s,\nu), \label{eqn:Lind}
\end{eqnarray}
by certain $\s_1 \in \I(GL_2(\A))$, $\nu \in \I(O_{K}(\A))$, and automorphic character $\chi$ of $\A^{\t}$ of weight $\k-3$.
However, if $F \in \Th_{2}(\s)$ satisfies the generalized Ramanujan conjecture at a good place $p$, then the $L$-function $L(s,F;r_5)$ of degree five is in the form:
\[
\big((1-p^{-s})(1-a_p p^{-s})(1-a_p^{-1} p^{-s})(1-b_p p^{-s})(1-b_p^{-1} p^{-s}) \big)^{-1}
\]
with $|a_p| = |b_p| =1$.
Since every automorphic form of $\Th_{2,2}(\s|_U)$ is holomorphic of weight $4$, comparing with (\ref{eqn:Lind}) for $\k = 4$, we obtain the last assertion.


\end{document}